\documentclass[11pt]{article}

\usepackage{amsmath,amssymb,amscd,amsthm,amsfonts}
\usepackage{graphicx,subfigure}
\usepackage{hyperref}
\usepackage{dsfont}
\usepackage{color}

\newtheorem{theorem}{Theorem}[section]
\newtheorem{theoremp}[theorem]{Theorem}
\newtheorem{lemma}[theorem]{Lemma}

\newtheorem{corollary}[theorem]{Corollary}
\newtheorem{conjecture}[theorem]{Conjecture}

\newtheorem{problem}[theorem]{Problem}

\newcommand{\rr}{\mathds{R}}

\def\rr{\mathds{R}}
\DeclareMathOperator{\conv}{conv}

\DeclareMathOperator{\depth}{depth}
\title{Robust Tverberg and colorful Carath\'eodory results via random choice}

\author{ Pablo Sober\'on}

\begin{document}

\maketitle

\begin{abstract}
	We use the probabilistic method to obtain versions of the colorful Carath\'eodory theorem and Tverberg's theorem with tolerance.
	
	In particular, we give bounds for the smallest integer $N=N(t,d,r)$ such that for any $N$ points in $R^d$, there is a partition of them into $r$ parts for which the following condition holds: after removing any $t$ points from the set, the convex hulls of what is left in each part intersect.

	  We prove a bound $N=rt+O(\sqrt{t})$ for fixed $r,d$ which is polynomial in each parameters.  Our bounds extend to colorful versions of Tverberg's theorem, as well as Reay-type variations of this theorem.
\end{abstract}

MSC2010 Classification:  52A35, 05D40

Keywords: Probabilistic method, Tverberg's theorem, Colorful Carath\'eodory theorem

\section{Introduction}

The colorful Carath\'eodory theorem and Tverberg's theorem are two gems of
 combinatorial geometry.  They describe properties of sets of points in $\rr^d$, each with a vast number of extensions and generalizations (for an introduction, see \cite{Matousek:2002td}).  The purpose of this paper is to show how an application of the probabilistic method yields robust versions of both results; i.e., the conclusions hold even if a small set of points is removed.

There are many applications of the probabilistic method in discrete geometry.  Among some notable examples are the crossing number theorem \cite{crossingnumber, Sze97}, the cutting lemma \cite{Cla87, CF90} and the existence of epsilon-nets for families of sets with bounded VC-dimension \cite{HW87}.  Further examples and an introduction to the general method can be found in \cite{Matousek:2002td, alonspencer}.

Let us begin with the colorful Carath\'eodory theorem, due to B\'ar\'any \cite{Barany:1982va}.  If we denote by $\conv(X)$ the convex hull of a set $X\subset \rr^d$, it says the following.

\begin{theoremp}[B\'ar\'any 1982]
Let $F_1, \ldots, F_{d+1}$ be $d+1$ families of points in $\rr^d$, considered as color clases.  If $0 \in \conv(F_i)$ for each $i$, there is a colorful choice $x_1 \in F_1, \ldots, x_{d+1} \in F_{d+1}$ such that $0 \in \conv\{x_1, x_2, \ldots, x_{d+1}\}$.
\end{theoremp}

Given a set $X \subset \rr^d$ and a point $p$, we say $X$ \textit{captures} $p$ if $p \in \conv(X)$.  We are interested in versions of the colorful Carath\'eodory theorem above  where the number of color classes is allowed to increase.  There are two natural variations of this kind, which we discuss in Section \ref{section-carath}.  Our main result is Lemma \ref{theorem-carath}, which shows the existence of a colorful choice which captures the origin even if any small subset of points is removed.

 Among the numerous consequences of the colorful Carath\'eodory theorem, there is a strickingly short proof of Tverberg's theorem by Sarkaria \cite{Sarkaria:1992vt}, later simplified by B\'ar\'any and Onn \cite{BaranyOnn}.  For a survey regarding Sarkaria's transformation, consult \cite{baranytensors}.

\begin{theoremp}[Tverberg 1966 \cite{Tverberg:1966tb}]
Given positive integers $d,r$ and $N=(d+1)(r-1)+1$ points in $\rr^d$, there is a partition of them into $r$ parts $A_1, \ldots, A_r$ such that
\[
\bigcap_{j = 1}^{r}\conv (A_j) \neq \emptyset.
\]
\end{theoremp}

One of the generalizations of this result, also known as Tverberg with tolerance, consists in finding partitions where the convex hulls of the parts intersect even after any $t$ points are removed.  Stated precisely, it says the following.

\begin{problem}[Tverberg partitions with tolerance]
Given positive integers $t,d,r$, find the smallest positive integer $N(t,d,r)$ such that for any set of $N(t,d,r)$ points in $\rr^d$, there is a partition of them into $r$ parts $A_1, A_2, \ldots, A_r$ such that for any set $C$ of at most $t$ points

\[
\bigcap_{j=1}^r \conv (A_j \setminus C) \neq \emptyset.
\]
\end{problem}

We refer to the parameter $t$ as the \textit{tolerance} of the partition.  The first bound for such partitions was given by Larman \cite{Larman:1972tn}, showing that $N(1,d,2) \le 2d+3$.  Larman's result is known to be optimal up to dimension four \cite{Forge:2001te}.  This was later improved by Garc\'ia-Col\'in to $N(t,d,2) \le (t+1)(d+1) +1$ \cite{GarciaColin:2007td, MR3386018}.  She also showed that $N(t,d,r) \le (t+1)(d+1)(r-1)+t+1$ for any triple $(t,d,r)$.  Garc\'ia-Col\'in conjectured a bound on $N$ extending her result, which was proven by Sober\'on and Strausz \cite{Soberon:2012er}, $N(t,d,r) \le (t+1)(d+1)(r-1)+1$.

The Sober\'on-Strausz bound is known not to be optimal, as shown by Mulzer and Stein for $d=1$ and in some instances for $d=2$ \cite{Mulzer:2013je}.  Recently, this was vastly improved by Garc\'ia-Col\'in, Raggi and Rold\'an-Pensado \cite{tolerance-new}, who showed that for fixed $r, d$ we have $N(t,d,r) = rt + o(t)$.  This settles the asymptotic behavior of $N(t,d,r)$ for large $t$, as the leading term matches the one for the lower bound $N(t,d,r) \ge rt+\frac{rd}{2}$, first given in \cite{soberon2015equal}.

However, the $o(t)$ term hides a $\operatorname{twr}_d ( O (r^2d^2))$ factor, where the tower functions $\operatorname{twr}_i(\alpha)$ are defined by $\operatorname{twr}_1 (\alpha) = \alpha$ and $\operatorname{twr}_{i+1} (\alpha) = 2^{\operatorname{twr}_i (\alpha)}$.  The tower function is unavoidable with the method they use, as it relies on geometric Ramsey-type results. Our main result is a new upper bound for $N(t,d,r)$.  In our bound, the leading term is also $rt$ for large $t$, and the bound is polynomial in $r,t,d$.

\begin{theorem}\label{theorem-tverberg}
For positive integers $t,d,r$, let $N(t,d,r)$ be the optimal number for Tverberg's theorem with tolerance.  Then, we have
\[
N(t,d,r) = rt + \tilde{O}(r^2\sqrt{td} + r^3d),
\]
where the $\tilde{O}$ notation hides only polylogarithmic factors in $t$, $d$ and $r$.
\end{theorem}

We should stress that the bound by Garc\'ia-Col\'in, Raggi and Rold\'an-Pensado is quite surprising by itself.  If are given less than $rt$ points, a trivial application of the pigeonhole principle shows that any partition of them into $r$ parts has one with at most $t$ points.  The removal of these points shows that the tolerance of any Tverberg partition is at most $t-1$.

In other words, with a large number of points the effect of the dimension on the combinatorics behind Tverberg's theorem with tolerance fades away.  Our result reinforces this counterintuitive claim by showing that, furthermore, one doesn't need to worry too much about the construction of the partition; a random one should suffice.  For other results in discrete geometry with tolerance, see \cite{Montejano:2011cg}.

Theorem \ref{theorem-tverberg} improves all previously known bounds when $t$ is large.  If $r > d$, it can be further improved to $N(t,d,r) = rt + \tilde{O}(rd\sqrt{rt} + r^2d^2)$ (see Theorem \ref{theorem-reay}), but both bounds remain $rt + \tilde O(\sqrt{t})$ for fixed $r, d$ with the same degree for the polynomial hidden by the $\tilde{O}$ notation.  The methods we use can be applied to yield versions with tolerance of several variations of Tverberg's theorem.  We exhibit this for two classic variations of Tverberg's theorem.  The first is the colored Tverberg theorem with tolerance, where the set of points and the desired partitions we want to obtain are given additional combinatorial conditions.  Our main result in this setting is the following theorem.

\begin{theorem}
Let $r \ge 3$.  Suppose we are given $(1.6) t + \tilde{O}(r\sqrt{td}+r^2d)$ families of $r$ points each in $\rr^d$, considered as color classes.  Then, there is a partition of them into $r$ sets $A_1, \ldots, A_r$, each with exactly one point of each color, such that even if any $t$ color classes are removed, the convex hulls of what is left in each $A_i$ intersect.  If $r=2$, the same result holds with $2 t + \tilde{O}(\sqrt{td}+d)$ families.  The $\tilde{O}$ notation only hides polylogarithmic factors in $r,t,d$.
\end{theorem}

The result above with a precise constant on the leading term is presented in Theorem \ref{colored-tverberg-padre}. 

 The second variation we present is related to Reay's conjecture.  Reay's conjecture is a relaxation of Tverberg's theorem.  The aim is, given a set of points in $\rr^d$, to find a partition of them into $r$ parts where the convex hulls of any $k$ parts intersect.  We call such partitions a Reay partition.  Tverberg partitions are those for which $k=r$.  It is an open question whether less points than those for Tverberg's theorem are needed to guarantee such a partition if $k<r$.  For the best bounds for this problem, see \cite{PS16, frickreu16}.
 
   In section \ref{section-reay} we show bounds for the number of points that guarantee the existence for Reay partitions with tolerance.  These bounds are smaller than those of Theorem \ref{theorem-tverberg}.  This is the first instance of a Reay-type result where the existence bounds are smaller than its Tverberg counterpart, albeit neither is known to be optimal.  We prove our Tverberg-type results in sections \ref{section-tverberg}, \ref{section-coloredtverberg} and \ref{section-reay}.

  We conclude by presenting remarks, open problems and algorithmic consequences of our results in section \ref{section-remarks}.
  
\section{Robust Carath\'eodory results}\label{section-carath}

The goal of this section is to extend the colorful Carath\'eodory theorem if we are given $N$ color classes instead of $d+1$.  We may try a direct approach and ask, given $N$ color classes, what conclusions can be reached if every color captures the origin.  Another option is, given $N$ color classes, to extend the contrapositive of the theorem and ask what happens if no colorful choice captures the origin.

For the latter case, a strengthening of the colorful Carath\'eodory theorem implies that, given $d+1$ color classes in $\rr^d$ such that no colorful choice captures the origin, there are \textit{two} colors $F_i, F_j$ whose union doesn't capture the origin (i.e., $0 \not\in \conv(F_i \cup F_j)$).  This was proved independently in \cite{Arocha:2009ft} and \cite{holmsen:2008id}.

\begin{theorem}\label{theorem-carath-old}
Let $N \ge d+1$ and $F_1, \ldots, F_{N}$ be sets of points in $\rr^d$, considered as color classes.  If no colorful choice captures the origin, there are $N-d+1$ colors whose union does not capture the origin.
\end{theorem}

Even though the result above does not follow from the colorful Carath\'eodory theorem, it can be proved with exactly the same arguments from \cite{Arocha:2009ft}, as noted by Imre B\'ar\'any \cite{bar-com}.  It was pointed later to the author that the Theorem \ref{theorem-carath-old} also follows from the main result of \cite{holmsen2016intersection}, which is an extends the colorful Carath\'eodory theorem to matroids.  We include below the proof following the arguments of \cite{Arocha:2009ft}.

\begin{proof}[Proof of Theorem \ref{theorem-carath-old}] 
We may assume without loss of generality that the set of points $\bigcup_i F_i$ is in sufficiently general position; i.e., there is no affine hyperplane spanned by $d$ of the points that contains the origin. Among all colorful choices $X$, there must be one, $X_0$ which minimizes $\operatorname{dist}(\conv(X_0),0)$.  Let $p$ be the closest point of $\conv(X_0)$ to the origin.  As $X_0$ does not capture the origin, $p$ must be in a $(d-1)$-face of $X_0$.  In other words, there must be at most $d$ points of different colors $x_1, \ldots, x_d$ whose convex hull contains $p$.

Let $H$ be the hyperplane  orthogonal to the vector $p-0$, which passes through $p$.  Let $H^+$ be the closed half-space of $H$ that does not contain the origin.  Note that $x_1, \ldots, x_d \in H^+$.  The minimality of $\operatorname{dist}(\conv(X_0),0)$ implies that the $N-d$ colors not containing $x_1, \ldots, x_d$ are contained in $H^+$.  If one of these remaining $d$ colors is also contained in $H^+$, we would have $N-d+1$ colors separated from the origin, as desired.

Let's assume that this does not happen, and look for a contradiction.  Thus, we can find $u_1, \ldots, u_d \in \rr^d \setminus H^+$ such that $u_i$ is of the same color of $x_i$.  Let $\ell_1$ be a ray starting from the origin in the direction of $p$.  Let $y$ be a point of some color which is not represented among the $x_i$.  We consider $\ell_2$ a ray starting from the origin in the direction of $-y$.  Notice that $\ell_2$ is contained in $\rr^d\setminus H^+$ since $y \in H^+$.

The colorful facets spanned by $\{x_1, \ldots, x_d, u_1, \ldots, u_d\}=V$ are the linear image of a $(d-1)$-dimensional octahedron onto $\rr^d$, which in turn is a continuous image of the sphere $S^{d-1}$.  This implies that these colorful facets must intersect the topological line $\ell_1 \cup \ell_2$ in an even number of points.  The fact that $\ell_1 \cup \ell_2$ does not intersect the facets in subfaces follows from the general position assumption.

As $p$ is one point of intersection, let's see where the other points can be.  In $H^+$ we can have only $p$, by the construction of $V$.  In the segment $[0,p)$ we can have no point of intersection, or we would contradict the minimality of $\operatorname{dist}(\conv(X_0),0)$.  In $\ell_2$ we can have no point of intersection, as the colorful $d$-tuple sustaining that point would capture the origin once $y$ is included.  This leads to the desired contradiction. 
\end{proof}

Now assume we are given $N$ color classes and each captures the origin.  We would like to obtain a colorful choice which does more than simply capturing the origin.  For this, we use the notion of depth.

Given a finite set $X \subset \rr^d$ and a point $p \in \rr^d$, we define
\[
\depth (X, p) = \min \{|H \cap X| : p \in H, \ H \ \mbox{is a closed half-space}\}.
\]

This is commonly known as Tukey depth of half-space depth \cite{tukey1975mathematics, liu2006data}.  A direct application of the colorful Carath\'eodory theorem shows that given $N$ color classes in $\rr^d$, each capturing the origin, there is a colorful selection $X$ such that $\depth(X,0) \ge \left\lfloor \frac{N}{d+1} \right\rfloor$, which is optimal.  However, in many applications of the colorful Carath\'eodory the color classes have few points compared to the dimension.  For this situation we obtain the following result.

\begin{lemma}\label{theorem-carath}
Let $F_1, F_2, \ldots, F_N$ be $N$ families of $r$ points each in $\rr^d$, considered as color classes.  If $0 \in \conv (F_i)$ for each $i$, we can make a colorful choice $x_1 \in F_1, \ldots, x_N \in F_N$ such that for the set $X = \{ x_1, \ldots , x_N \}$ we have
\[
\depth(X,0) \ge \frac{N}{r} - \sqrt{\frac{d N \ln (Nr)}{2}}
\]
\end{lemma}

For the proof of Lemma \ref{theorem-carath} we need the following lemma, mentioned in \cite{Clarksonradon}, for instance.

\begin{lemma}\label{lemma-dimension}
Given a set $Y$ of $M$ points in $\rr^d$ and a point $c \in \rr^d$, there is a family of $M^d$ closed half-spaces containing $c$ such that, for any subset $Z \subset Y$, if each of the half-spaces contains at least $t$ points of $Z$, then we have $\depth (Z, c) \ge t$, for any $t$.
\end{lemma}

\begin{proof}
	We may assume that the affine span of $Y \cup \{c\}$ is $\rr^d$.  Otherwise, there is a hyperplane containing $Y \cup \{c\}$ and we can apply this Lemma for a lower dimension.  To check the $\depth (Z, c)$, it suffices to check half-spaces whose defining hyperplane goes through $c$.   Given a closed half-space $H^+$ such that $c \in H$, consider the set $\mathcal{M} = H^+\cap Y$.  If we show that there are at most $M^d$ possible subsets $\mathcal{M}$ we gan get this way, then the lemma would be proved.
	
	Notice that we may move continuously $H$ without losing $c$ until it contains $d-1$ points $y_1, \ldots, y_d \in Y$ such that $c, y_1, \ldots, y_{d-1}$ are affinely independent and without changing the set $M$ with the only possible exception of gaining a subset of $y_1, \ldots, y_{d-1}$.  Therefore, there are at most $2 \cdot 2^{d-1} \cdot {{M}\choose{d-1}}$ possibilities for $M$.  This number comes from counting the choice of the $(d-1)$-tuple of $Y$ to generate $H$ after tilting, the $2^{d-1}$ is the possible number of subsets of $y_1, \ldots, y_{d-1}$ that could be in $\mathcal{M}$ and the first factor $2$ comes from the two possible half-spaces we can be considering once we have $H$ in its final position.  This numbers is bounded above by $M^d$ except for the case $M = 3, d=2$, which can be checked by hand.
\end{proof}

The proof of Lemma \ref{theorem-carath} uses the probabilisitic method, by making the colorful choice at random.

\begin{proof}[Proof of Lemma \ref{theorem-carath}]
Let $\lambda > \sqrt{\frac{d N \ln (Nr)}{2}}$.  For each $F_i$ we are going to choose $x_i$ randomly and uniformly from its $r$ points in order to form our set $X$.

Given a closed half-space $H$ such that $0 \in H$, since $0 \in \conv (F_i)$, we have that 
\[
\mathbb{P} ( x_i \in H) \ge \frac{1}{r}.
\]

Thus, 

\[
\mathbb{E} ( | X \cap H |  ) \ge \frac{N}{r}.
\]

Moreover, $|X \cap H| = \sum_{i=1}^N \chi (x_i \in H)$.  This is a sum of independent indicators each of whose probability of success is at least $\frac{1}{r}$. In particular, Hoeffding's inequality implies that

\[
\mathbb{P} \left( |X \cap H| \le \frac{N}{r} - \lambda\right) \le \mathbb{P} ( |X \cap H| \le \mathbb{E}(|X \cap H|) - \lambda) \le \exp \left(\frac{-2\lambda^2}{N}\right).
\]
We can denote by $X_H$ the random variable which is the indicator of the event $|X \cap H| \le \frac{N}{r} - \lambda$.

Let $\mathcal{H}$ be the family of at most $(Nr)^d$ half-spaces from Lemma \ref{lemma-dimension} with $Y = \cup_i F_i$ and $c = 0$.  We have that

\[
\mathbb{E} \left( \sum_{H \in \mathcal{H}} X_H\right) =  \sum_{H \in \mathcal{H}}\mathbb{E}(X_H) \le \sum_{H \in \mathcal{H}} \exp \left(\frac{-2\lambda^2}{N}\right) \le \exp({d\ln (Nr)})\exp \left(\frac{-2\lambda^2}{N}\right) < 1,
\]
where the last inequality follows from the choice of $\lambda$.  Thus, there must be an instance of $X$ where $\sum_{H \in \mathcal{H}} X_H = 0$.  By the choice of $\mathcal{H}$, we have that $\depth(X,0) \ge \frac{N}{r}-\lambda$, as desired.
\end{proof}

One curious aspect of the colorful Carath\'eodory theorem is the case $r=2$.  In this instance, the color classes are simply the endpoints of $d+1$ segments each containing the origin.  To show the existence of a colorful choice as the theorem indicates, it suffices to use a non-trivial linear dependence of the $d+1$ directions of the segments.  The signs of the coefficients indicate which endpoint should be taken for each segment.  Likewise, Lemma \ref{theorem-carath} has a similar version when $r=2$.

\begin{corollary}
	Let $N, d$ be positive integers and $t =  \left\lceil \frac{N}{2}-\sqrt{\frac{d N \ln (2N)}{2}}\right\rceil -1$.  Given a set $S$ of $N$ vectors in $\rr^d$ there is an assignment $\Gamma$ of signs $(+)$ or $(-)$ to the element of $S$ such that, if any $t$ vectors are removed from $S$, the remaining vectors have a non-trivial linear dependence where the signs of all nonzero coefficients agree with $\Gamma$.
\end{corollary}

\section{Robust Tverberg results}\label{section-tverberg}

Let us restate Theorem \ref{theorem-tverberg} in the form which we aim to prove.  In order to show that the version below implies the one in the introduction it suffices to write $N = rt + p$ and notice that with $p = \tilde{O}(r^2\sqrt{td} + dr^3)$, the tolerance provided by the result below is at least $t$.

\begin{theorem}
Let $X$ be a set of $N$ points in $\rr^d$, and let

\[
t = \left\lceil \frac{N}{r} - \sqrt{\frac{(d+1)(r-1) N \ln (Nr)}{2}}\right\rceil - 1.
\]
Then, there is a partition of $X$ into $r$ parts $A_1, A_2, \ldots, A_r$ such that for any set $C$ of at most $t$ points, we have
\[
\bigcap_{j=1}^r  \conv(A_j \setminus C)  \neq \emptyset.
\]
\end{theorem}

Since Tverberg's theorem can be deduced using the colorful Carath\'eodory theorem, it should come as no surprise that variations of the latter often translate to variations of both.  The result above is the consequence of applying Lemma \ref{theorem-carath}.

\begin{proof}
Let $S$ be a set of $N$ points in $\rr^d$, $S = \{ a_1, \ldots, a_N \}$.  Let $u_1, u_2, \ldots, u_r$ be the vertices of a regular simplex in $\rr^{r-1}$ centered at the origin.  Notice that any linear combination $\beta_1 u_1 + \ldots + \beta_r u_r$ gives the zero vector if and only if $\beta_1 = \ldots = \beta_r$.

We construct the points $b_i = (a_i, 1) \in \rr^{d+1}$ and for each $i$ consider the family $F_i = \{ b_i \otimes u_j : 1 \le j \le r\} \subset \rr^{(d+1)(r-1)}$, where $\otimes$ denotes the standard tensor product.  Notice that the barycenter of each $F_i$ is the origin in $\rr^{(d+1)(r-1)}$.  Thus, we may apply Lemma \ref{theorem-carath} and obtain a colorful choice $x_i = b_i \otimes u_{j_i}$ for each $i$ such that for the set $X = \{x_1, \ldots, x_N \}$ we have $\depth(X, 0) \ge t+1$.  If we remove any set $C$ of at most $t$ points, we still have $\depth(X\setminus C, 0) \ge 1$.  In other words, $0 \in \conv (X \setminus C)$.

Consider the sets $I_j = \{i: j_i = j \}$ and $A_j = \{u_i : i \in I_j \}$.  The sets $A_1, A_2, \ldots, A_r$ form the partition of $S$ induced by $X$.  Given any set $C \subset \{1,\ldots, N\}=[N]$ of at most $t$ indices, we know that there are coefficients $\{\alpha_i : i \in [N]\setminus C\}$ of a convex combination such that
\[
\sum_{i \in [N]\setminus C} \alpha_i b_i \otimes u_{j_i} = 0.
\] 
If we factor each $u_j$ we get
\[
\left( \sum_{i \in I_1 \setminus C} \alpha_i b_i \right) \otimes u_1 + \left( \sum_{i \in I_2 \setminus C} \alpha_i b_i \right) \otimes u_2 + \ldots + \left( \sum_{i \in I_r \setminus C} \alpha_i b_i \right) \otimes u_r = 0.
\]
The choice of $u_1, \ldots, u_r$ implies that
\[
\sum_{i \in I_1 \setminus C} \alpha_i b_i = \ldots = \sum_{i \in I_r \setminus C} \alpha_i b_i
\]
Using the fact that the last coordinate of each $b_i$ is equal to $1$, by a simple scaling we can assume that $\sum_{i \in I_j \setminus C} \alpha_i = 1$ for each $j$. If we look at the first $d$ coordinates, we have
\[
\sum_{i \in I_1 \setminus C} \alpha_i a_i = \ldots = \sum_{i \in I_r \setminus C} \alpha_i a_i,
\]
where each expression is a convex combination.  If we define $C' = \{u_i: i \in C\}$ the convex combinations above translate to
\[
\bigcap_{j=1}^r \conv( A_j \setminus C') \neq \emptyset,
\]
as desired.
\end{proof}

\section{Colored Tverberg with tolerance}\label{section-coloredtverberg}

One of the most important open problems around Tverberg's theorem is a long-standing conjecture by B\'ar\'any and Larman, aslo known as the colored Tververg theorem \cite{Barany:1992tx}.

\begin{conjecture}
Suppose $F_1, \ldots, F_{d+1}$ are families of $r$ points each in $\rr^d$, considered as color classes.  Then, there is a partition of them into $r$ sets $A_1, A_2, \ldots, A_r$ such that each part has exactly one point of each color and 
\[
\bigcap_{j=1}^r \conv (A_j) \neq \emptyset
\]
\end{conjecture}

This conjecture has only been verified for $d=2$ by B\'ar\'any and Larman \cite{Barany:1992tx}, whose paper also included a proof for the case $r=2$ by Lov\'asz, and when $r+1$ is prime \cite{Blagojevic:2011vh, blago15}.  Some relaxations which give a positive result are if we are given $(r-1)d+1$ color classes instead of $d+1$ \cite{soberon2015equal}, where we can impose further conditions on the coefficients of the intersecting convex combinations; and if we allow each color class to have $2r-1$ points instead of $r$ \cite{Blagojevic:2014js}, although in this case the sets $A_i$ do not form a partition.

Given a family of sets $F_j$ of $r$ points each, considered color classes, we will say that a partition of their union into $r$ sets $A_1, \ldots, A_r$ is a colorful partition if each $A_i$ has exactly one point of each $F_j$.  We show a version of Tverberg's theorem with tolerance which holds for colored classes.  

\begin{theorem}\label{colored-tverberg-padre}
There is a constant $c_r$ such that the following hold.  Given $c_rt+\tilde{O}(r\sqrt{td}+rd)$ color classes of $r$ points each in $\rr^d$, there is a colorful partition of them into $r$ sets $A_1, \ldots, A_r$ such that, even if we remove any $t$ color classes, the convex hull of what is left in each $A_i$ still intersect.  The $\tilde{O}(\cdot)$ notation only hides polylogarithmic factors in $r,t,d$.  Moreover, $c_r \to \frac{e}{e-1}\sim 1.581...$ as $r \to \infty$.
\end{theorem}

It should be noted that the adaptation of Sarkaria's methods described in \cite{soberon2015equal} in combination of Theorem \ref{theorem-tverberg} does give a colored version similar to the one above.  This would have much stronger conditions on the coefficients of the convex combinations that give the point of intersection, but would require $rt + o(t)$ color classes instead of $c_r t + o(t)$.  The value $c_r$ also satisifes the bounds $c_2 \le 2$ and $c_r \le 1.6$ for $r \ge 3$.  To prove Theorem \ref{colored-tverberg-padre}, we need the following lemmatta.

\begin{lemma}
Consider $[r] = \{1,2, \ldots, r\}$.  Suppose that we are given $r$ forbidden values $v_1, v_2, \ldots, v_r$ not necessarily distinct in $[r]$.  The probability that a random permutation $\sigma$ satisfies $\sigma (i) \neq v_i$ for all $i$ is maximized when all the $v_i$ are different. 
\end{lemma}

\begin{proof}
Let $S_r$ be the set of all permutations $\sigma: [r] \to [r]$.  For each permutation $\sigma \in S_r$, let $n(\sigma)$ be the number of indices $i$ such that $\sigma (i) = v_i$.  If the forbidden values are not all different, we can assume without loss of generality that $v_1 = v_2$.  Thus, there is at least one number $\tau \in [r]$ which is not forbidden.  Consider a new list of forbidden values $(\tau, v_2, \ldots, v_r)$, and let $m(\sigma)$ be the number of indices $i$ such that $\sigma(i)$ is the $i$-the element of the new list.  Let $F: S_r \to S_r$ be the bijection such that $F(\sigma) = \sigma$ if $\sigma (1) \neq v_1$ and $\sigma(1) \neq \tau$, and $F$ switches the values of $\sigma^{-1} (v_1$) and $\sigma^{-1}(\tau)$ otherwise.  If $n(\sigma) = 0$, then $m(F(\sigma)) = 0$, but there may be permutations with $n(\sigma) \neq 0$ and $m(F(\sigma)) = 0$.  Thus, if we repeat this process until every pair of forbidden values is different, the number of permutations $\sigma$ with $n (\sigma) = 0$ only increases, as desired.
\end{proof}

It is known that the probability $p(r)$ that a random permutation of $[r]$ has at least one fixed point tends to $1- \frac{1}{e}$ as $r \to \infty$.  Thus, the lemma above implies that, given at least one forbidden value for each element in $[r]$, the probability that a random permutation of $[r]$ hits at least one of them is at least $p(r)$.

We say that a family of $r$ sets $F_1, \ldots, F_r$ in $\rr^d$ is a colored $r$-block if 
\begin{itemize}
	\item each $F_i$ has $r$ points,
	\item each $F_i$ captures the origin and
	\item all points in the $r$-block are colored with one of $r$ possible colors in such a way that each $F_i$ has exactly one point of each color.
\end{itemize}
Given a colored $r$-block, we say that a subset of its points is a colorful choice if it has exactly one point of each color and exactly one point of each $F_i$.  Given a family of colored $r$-blocks, we say a subset of their union is a colorful choice if its restriction to each $r$-block is also a colorful choice.

\begin{corollary}
Suppose that we are given a colored $r$-block in $\rr^d$ and a closed half-space $H$ containing the origin.  If we pick a random colorful choice, the probability that we have at least one point in $H$ is at least $p(r)$.  
\end{corollary}

\begin{proof}
Let $F_1, \ldots, F_r$ be the sets in the block.  It suffice to notice that each colorful choice corresponds to a permutation $\sigma: [r] \to [r]$ so that $\sigma (i) = j$ if and only if the point chosen from $F_i$ is of color $j$.  Then, if we forbid all points in $H$ (at least one forbidden point from each $F_i$), we have reduced the problem to the lemma above.
\end{proof}

\begin{theorem}\label{theorem-block-caratheodory}
Suppose we are given $N$ colored $r$-blocks in $\rr^d$, all using the same $r$ colors.  Then, there is a colorful choice $M$ such that each half-space containing the origin contains points of at least $t$ different $r$-blocks, as long as
\[
t \le p(r) N - \sqrt{\frac{dN\ln (Nr^2)}{2}}.
\]
\end{theorem}

\begin{proof}
Let $\lambda > \sqrt{\frac{dN\ln (Nr^2)}{2}}$ and $H$ be a closed half-space containing the origin.  For each colored $r$-block $B$ we choose independently at random a colored choice $X_B$, their union is a random colored choice $X$ for the whole family.  Let $x_B = \chi (X_B \cap H \neq \emptyset)$.  Then
\begin{itemize}
	\item $\mathbb{E}(x_B) \ge p(r)$,
	\item $\sum_B x_B$ is the number of colored $r$-blocks whose colorful choice has at least one point in $H$, so $\mathbb{E}(\sum_B x_B) \ge p(r) N$,
	\item for $B \neq B'$, $x_B$ and $x_{B'}$ are independent.
\end{itemize}
Thus, after applying Hoeffding's inequality we get
\[
\mathbb{P}\left(\sum_B x_B \le p(r) N - \lambda \right) \le \exp \left( \frac{-2\lambda^2}{N}\right).
\]
Notice that among all the colored $r$-blocks we have $Nr^2$ points.  Thus, to find the depth of $0$ from a subset it is sufficient to check a family $\mathcal{H}$ of at most $(Nr^2)^d$ half-spaces.  If we call a half-space $H$ \textit{bad} if less than $p(r) N - \lambda$ of the $r$-blocks have a point in it, the probability that there is at least one bad half-space is at most
\[
\sum_{H \in \mathcal{H}} \exp \left( \frac{-2\lambda^2}{N}\right) = (Nr^2)^d \exp \left( \frac{-2\lambda^2}{N}\right) = \exp \left( d \ln(Nr^2) -\frac{2\lambda^2}{N}\right) < 1.
\]
Thus, there is at least one colorful choice with no bad hyperplanes.

\end{proof}

Now we are ready to prove Theorem \ref{colored-tverberg-padre}.

\begin{proof}[Proof of Theorem \ref{colored-tverberg-padre}]
We will show that given $N$ color classes of $r$ points each in $\rr^d$, we can find a colorful partition of them into $r$ sets $A_1, \ldots, A_r$ so that, even if we remove any $r$ color classes, the convex hulls of what is left in each $A_i$ still intersect as as long as
\[
t \le  p(r) N - \sqrt{\frac{(d+1)(r-1)N\ln (Nr^2)}{2}}-1.
\]
This implies the result in the theorem with $c_r = \frac{1}{p(r)}$.

Let $n = (r-1)(d+1)$.  If we apply the Sarkaria transformation to the original set of points, notice that each color class of $r$ points is represented by $r^2$ points in $\rr^n$.  Moreover, we can introduce $r$ colors in $\rr^n$ and paint the point of the form $(x,1) \otimes u_i$ of color $i$.  This turns the $r^2$ points in $\rr^n$ into a colored $r$-block.  Notice also that for a color class in $\rr^d$, assigning each of its points to a different $A_i$ corresponds to a colorful choice in its corresponding colored $r$-block in $\rr^n$.  Thus, we can apply Theorem \ref{theorem-block-caratheodory} to finish the proof.
\end{proof}

\section{Reay partitions with tolerance}\label{section-reay}

Let $R^*=R^*(d,r,k)$ be the smallest integer such that among any $R^*$ points in $\rr^d$, there is a partition of them into $r$ parts such that the convex hulls of any $k$ parts intersect.  We call such partitions \textit{Reay partitions}.  Tverberg's theorem asserts that $R^*(d,r,r) = (d+1)(r-1)+1$.  However, there are no cases known for which $R^*(d,r,k) < R^*(d,r,r)$.  In 1979 Reay conjecture that $R^*(d,r,k) = (d+1)(r-1)+1$ for any $r\ge k \ge 2$ \cite{reay1979several}.  Reay's conjecture remains open.  There have been several advances improving lower bounds for $R^*$ \cite{PS16, frickreu16}.  

It is natural to extend Reay partitions to the setting with tolerance.  We define the integer $R=R(t,d,r,k)$ as the smallest $R$ such that among any $R$ points in $\rr^d$, there is a partition $\mathcal{P}$ of them into $r$ parts with the following property.  For any set $C$ of at most $t$ points and any $k$-tuple $\mathcal{K}$ of parts of $\mathcal{P}$, even if the points of $C$ are removed, the convex hulls of what is left in each part of $\mathcal{K}$ intersect.

\begin{theorem}\label{theorem-reay}
For any positive integers $t,d,r,k$ with $r \ge k$, we have that
\[
R(t,d,r,k) = rt + \tilde{O}\left(r\sqrt{dkrt}+r^2 dk  \right).
\]
\end{theorem}

Moreover, an application of Helly's theorem show that $N(t,d,r) = R(t,d,r,d+1)$, so the bounds above improve Theorem \ref{theorem-tverberg} if $r$ is large.

\begin{proof}
	Suppose we are given $M$ points in $\rr^d$.  We can color them randomly and independently with one of $r$ colors.  Let us bound the probability that a given $k$-tuple of colors can be separated if we remove $t$ points.  In other words, we want to bound from above the probability that there is a set of at most $t$ points such that after removing them, the convex hulls of what is left in each part of the $k$-tuple do not intersect.
	
	We apply Sarkaria's trick, but use only $k$ vectors $u_1, \ldots, u_k$ in $\rr^{k-1}$ instead of $r$ vectors in $\rr^{r-1}$.  Every point $a_i \in \rr^d$ is represented by the set $F_i$ made by the $k$ points of the form $(a_i, 1) \otimes u_j$ for $1 \le j \le k$.  If our chosen $k$-tuples consists of the first $k$ colors, then assigning a color to $a_i$ corresponds to possibly choosing an element of $F_i$ as follows:
	
	\begin{itemize}
		\item if $a_i$ is colored with color $j$ and $j \le k$, then we choose $(a_i, 1) \otimes u_j$;
		\item if $a_i$ is colored with some color $j \ge k+1$, we don't make a choice from $F_i$.
	\end{itemize}
	  This turns our partition in $\rr^d$ into a colorful choice $X$ in $\rr^{(k-1)(d+1)}$, where some classes $F_i$ may not have an element selected.
	
	The tolerance with which the convex hulls of our $k$-tuple intersect is equal to the $\depth (X, 0) -1$.  With similar computation as those for the proof of Theorem $1$, we get that for any $\lambda > 0$,
	\[
	\mathbb{P} \left(\depth (X,0) < \frac{M}{r}- \lambda \right) \le \exp\left((d+1)(k-1)\ln (Mr) - \frac{2\lambda^2}{M}\right).
	\]
	So, the probability that there is a $k$-tuple with tolerance smaller than $\frac{M}{r}-\lambda-1$ is at most
	\[
	{{r}\choose{k}}\exp\left((d+1)(k-1)\ln (Mr) - \frac{2\lambda^2}{M}\right)
	\]
	
	Thus, by choosing $\lambda > \sqrt{\frac{1}{2}\left[(d+1)(k-1)\ln (Mr) + \ln{{r}\choose{k}}\right]}$ we know there is an instance where this does not happen.  If we check how large $M$ must be to guarantee that the given tolerance is at least $t$, we get the asymptotic bound of the theorem.  One should note that $\ln {{r}\choose{k}} \le k \ln \left( \frac{r}{k }\right)= \tilde O (k)$, which reduces the number of terms we get in the final expression.
\end{proof}

We should stress that the lower bound mentioned in the introduction, $N(t,d,r) \ge rt+\frac{rd}{2}$, extends to Reay's setting.  This is because the construction has the property that for every partition into $r$ parts, there is a set of $t$ points such that their removal makes one of the parts to separate by a hyperplane form the rest of the set.  In other words,
\[
N(t,d,r) = R(t,d,r,r) \ge R(t,d,r,k) \ge R(t,d,r,2) \ge rt+\frac{rd}{2}.
\]

\section{Remarks}\label{section-remarks}

It would be interesting to see if in Theorem \ref{colored-tverberg-padre}, the constant $c_r \sim 1.582...$ could be replaced by $1$.  Namely, determining if the following result holds.

\begin{conjecture}\label{conjecture-fuerte}
Let $r, d$ be fixed positive integers.  There is an integer $M=M(t,d,r)= t(1+o(1))$ such that, given any $M$ families of $r$ points each in $\rr^d$, there is a colorful partition of them $A_1, \ldots, A_r$ such that for any family $C$ of at most $t$ colors
\[
\bigcap_{j=1}^r \conv(A_j \setminus C) \neq \emptyset 
\] 
\end{conjecture}

This may seem counterintuitive at first sight, as we are removing almost all the points.  One of the novel methods to obtain colorful Tverberg results is the ``constraints'' method by Blagojevi\'c, Frick and Ziegler \cite{Blagojevic:2014js}.  It is tempting to use Theorem \ref{theorem-tverberg} with the constraint method to tackle the conjecture above.  However, since the tolerance given by Theorem \ref{theorem-tverberg} is at most a $(1/r)$-fraction of the total number of points, Conjecture \ref{conjecture-fuerte} seems out of reach.

  Theorem \ref{theorem-tverberg} doesn't improve the previous bounds for $N(t,d,r)$ for low values of $t$.  It is still possible that Larman's result is optimal.

\begin{problem}[Larman 1972]
Determine if there is a set $S$ of $2d+2$ points in $\rr^d$ with the property that, for any partition $A,B$ of $S$, there is a point $x \in S$ such that
\[
\conv(A\setminus \{x\}) \cap \conv (B \setminus \{ x \}) = \emptyset.
\]
\end{problem}

Finally, the topological versions of Tverberg's theorem with tolerance remain open, even in the cases where $r$ is a prime number or a prime power. 

\begin{problem}
Given an integer $n$, denote by $\Delta^n$ the $n$-dimensional simplex, with $n+1$ vertices.  Find the smallest integer $N^*= N^*(t,d,r)$ such that the following holds.  Given any continuous map $f: \Delta^{N^*} \to \rr^d$ there is a partition of the vertices of $\Delta^{N^*}$ into $r$ sets $A_1, \ldots, A_r$ such that for any set $C$ of at most $t$ vertices
\[
\bigcap_{j=1}^r f([A_j \setminus C]) \neq \emptyset,
\] 
where $[X]$ denotes the face spanned by $X$, for any set of vertices $X$.
\end{problem}

So far, not even the Sober\'on-Strausz bound $N^*\le (t+1)(r-1)(d+1)$ is known to hold.  The only bound at the moment is $N^* \le (t+1)(r-1)(d+1)+t$, when $r$ is a prime power.  This follows from taking $t+1$ topological Tverberg partitions, as removing $t$ vertices leaves one of the partitions unaffected.

If one is interested in non-deterministic algorithms that yield Tverberg theorems with tolerance, the proof of our results can be extended to the following.

\begin{theorem}
Let $N,t,d,r$ be positive integers and $\varepsilon >0$ be a real number.  Given $N$ points in $\rr^d$, a random partition of them into $r$ parts is a Tverberg partition with tolerance $t$ with probability at least $1 - \varepsilon$ as long as 
\[
t + 1 \le \frac{N}{r}- \sqrt{\frac12 \left[(d+1)(r-1) N \ln (Nr) + N \ln \left( \frac{1}{\varepsilon}\right)\right]}.
\]
\end{theorem}

The advantage of the result above is that generating the partition is trivial, taking time $N$.  The problem of finding Tverberg partitions, in both its deterministic and non-deterministic version, is interesting.  See, for instance, \cite{Clarksonradon, Miller:2009bg, Mulzer:2013je, algorithmic-new}.




\bibliographystyle{amsalpha}

\bibliography{references}

\noindent Pablo Sober\'on \\
\textsc{
Mathematics Department \\
Northeastern University \\
Boston, MA 02445
}\\[0.1cm]

\noindent \textit{E-mail address: }\texttt{p.soberonbravo@northeastern.edu}

\end{document}